\documentclass{gtpart}
\usepackage{amssymb}
\usepackage{enumerate}   
\usepackage{graphicx}
\usepackage{tikz-cd} 
\usepackage{color}
\usepackage{hyperref}

\newtheorem{theorem}{Theorem}

\newtheorem{conjecture}{Conjecture}
\newtheorem{remark}{Remark}
\newtheorem{lemma}{Lemma}

\def\Zbb{\mathbb{Z}}
\def\Rbb{\mathbb{R}}
\def\Cbb{\mathbb{C}}
\def\Fbb{\mathbb{F}}

\def\Mt{\widetilde{M}}

\newcommand{\G}[1]{
	\mathrm{SL}_{#1}(\mathbb{C})
}

\begin{document}
	
	\title{A vanishing identity on adjoint Reidemeister torsions \\of twist knots}
	

	\author{Seokbeom Yoon}
	\email{sbyoon15@kias.re.kr}

	
	
	
	
	\begin{abstract} 
		For a compact oriented  3-manifold with torus boundary the adjoint Reidemeister torsion is defined as a function on the $\mathrm{SL}_2(\mathbb{C})$-character variety  depending on a choice of a boundary curve.
		Under reasonable assumptions, it is conjectured that the adjoint torsion satisfies a certain type of vanishing identities.
		In this paper, we prove that the conjecture holds for all hyperbolic twist knot exteriors by using Jacobi's residue theorem. 
	\end{abstract}
	
	\maketitle
	
	\section{Introduction}
	\subsection{Overview}
	
	Let $M$ be a compact oriented 3-manifold with torus boundary and let $\mathcal{X}^\mathrm{irr}(M)$ be the character variety of irreducible representations $\pi_1(M)\rightarrow \G{2}$. 
	We assume that every irreducible component of $\mathcal{X}^\mathrm{irr}(M)$ is of dimension 1.
	Note that there are many known examples satisfying the assumption: for instance, whenever $M$ contains no closed incompressible surface \cite[\S 2.4]{cooper1994plane}.
	
	In \cite{porti1997torsion} Porti defined the \emph{adjoint torsion}, denoted by $\tau_\gamma$, as a function on a Zariski open subset of $\mathcal{X}^\mathrm{irr}(M)$ depending on a choice of a boundary curve $\gamma$. Here a boundary curve means a simple closed curve in $\partial M$ with a non-trivial class in $H_1(\partial M;\Zbb)$.
	Roughly speaking, at the character $\chi_\rho$ of an irreducible representation $\rho:\pi_1(M)\rightarrow \G{2}$ the value $\tau_{\gamma}(\chi_\rho)$  is the sign-refined Reidemeister torsion twisted by the adjoint action associated to $\rho$. 
	The definition involves the choice of a boundary curve $\gamma$ so as to specify a basis of the twisted (co-)homology.
	We briefly recall the definition in Section \ref{sec:rev}.

	Since Witten's monumental paper \cite{witten1989quantum}, there have been various studies on adjoint torsion in terms of quantum field theory. 
	Several recent studies \cite{gang2019precision, benini2019rotating, gang2019adjoint} on the relation with the Witten index suggested the following conjecture.
	\begin{conjecture}\label{conj:main}
		Suppose that every component of $\mathcal{X}^\mathrm{irr}(M)$ is of dimension 1 and that the interior of $M$ admits a hyperbolic structure of finite volume.
		Then for any boundary curve $\gamma \subset \partial M$ we have
		\begin{equation*}
		\sum_{\chi_\rho \in \mathrm{tr}_\gamma^{-1}(C)} \frac{1} {\tau_{\gamma}(\chi_\rho)} =0
		\end{equation*}
		for generic $C\in\Cbb$ where $\mathrm{tr}_\gamma : \mathcal{X}^\mathrm{irr}(M)\rightarrow \Cbb$ is the trace function of $\gamma$.
	\end{conjecture}
	
%
	A knot in $S^3$ with a diagram as in Figure \ref{fig:twist_knot} is called a \emph{twist knot}. We denote by $K_n$ for $n \neq 0 \in \Zbb$ the twist knot having $|n|$ right-handed half twists in the box (left-handed, if $n$ is negative). We may focus on twist knots $K_{2n}$ since $K_{2n+1}$ is equivalent to the mirror image of $K_{-2n}$.
	\begin{figure}[!h]
		\begin{center}
			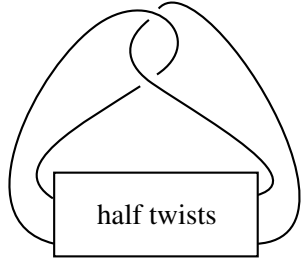
			\caption{A diagram of a twist knot.}
			\label{fig:twist_knot}
		\end{center}
	\end{figure}
	
	It is known that every knot exterior of $K_{2n}$  satisfies the assumptions in Conjecture~\ref{conj:main} except for $n=1$  (the trefoil knot).
	Namely, every twist knot $K_{2n}$ with knot exterior $M$ is hyperbolic except for $n=1$   \cite{menasco1984closed}  and has the character variety $\mathcal{X}^\mathrm{irr}(M)$ consisting of 1-dimensional components.
	The latter can be derived from explicit computations \cite{riley1984nonabelian,macasieb2011character} or the fact that every twist knot exterior contains no closed incompressible surface \cite{hatcher1985incompressible}. Main aim of this paper is to prove that Conjecture \ref{conj:main} holds for all hyperbolic twist knots.

	\begin{theorem} \label{thm:main} Let $M$ be the knot exterior of the twist knot $K_{2n}$ for $n \neq 0,1$ and  $\gamma \subset \partial M$ be a boundary curve. Then we have 
		\begin{equation*}
		\sum_{\chi_\rho \in \mathrm{tr}_\gamma^{-1}(C)} \dfrac{1}{\tau_\gamma(\chi_\rho)} =0
		\end{equation*}
		for generic $C \in \Cbb$.
	\end{theorem}
	
	The adjoint torsion is quite hard to compute in general and its concrete computation is known only in a few examples.
	For twist knots the adjoint torsion $\tau_\lambda$ with respect to the canonical longitude $\lambda$ was first computed in \cite{dubois2009non} by using the relation with the twisted Alexander polynomial \cite{yamaguchi2008relationship} and the computation were remarkably simplified in \cite{tran2014twisted} (also in \cite{tran2018twisted}). 
	To prove Theorem \ref{thm:main}, we extend the computation  to an arbitrary boundary curve.
	We summarize our computation here for convenience of the reader. 
	We refer to Section \ref{sec:comp} for details.
	
	Let $S_k(z)$ be the Chebyshev polynomials defined by $S_{k+1}(z) = z S_k(z) - S_{k-1}(z)$ for all $k \in \Zbb$ with the initial condition $S_0(z)=0$, $S_1(z)=1$. Throughout the paper, the Chebyshev polynomials are always given in the variable $z$ and thus we often write $S_k(z)$ simply as $S_k$. 
	We fix $n \neq 0 \in \Zbb$ and let $M$ be the knot exterior of the twist knot $K_{2n}$.
	Reformulating \cite{riley1984nonabelian}, the character variety $\mathcal{X}^\mathrm{irr}(M)$ is given by the zero set of
	\[F(m,z)=S_{n} (S_{n}-S_{n-1})(m^2+m^{-2})-zS_{n} (S_{n}-S_{n-1})+1\]
	in $\Cbb^\ast \times \Cbb$ with the quotient identifying $(m,z)$ and $(m^{-1},z)$.
	For a boundary curve $\gamma \subset \partial M$ we let
	\begin{equation*}
	E_\gamma(m,z) = m^p \left( -\dfrac{(z-2) (S_{n+1}-S_{n-1})S_{n}^2}{S_{n}-S_{n-1}} m^2+ (z-2)(S_{n}+S_{n-1})S_{n}+1\right)^q
	\end{equation*}
	where $p/q \in \mathbb{Q} \cup \{\infty\}$ is the slope of $\gamma$.
	\begin{theorem} \label{thm:main2} Let $(m,z)\in\Cbb^\ast \times \Cbb$ be a solution to the Laurent polynomial $F$ and $\chi_\rho$ be the corresponding character. Then $E_\gamma(m^{\pm1},z)$ are  eigenvalues of $\rho(\gamma)$ and
		\begin{equation}\label{eqn:main2}
		\tau_\gamma(\chi_\rho) = -\dfrac{m}{2E_\gamma} \det \left( \dfrac{\partial(F,E_\gamma)}{\partial(m,z)}\right) 
		\end{equation} if $\chi_\rho$ is $\gamma$-regular. 
	\end{theorem}
	\begin{remark}
		Choosing a boundary curve $\gamma$ as a meridian $\mu$, we have $E_\gamma =m$ and  the equation \eqref{eqn:main2} reduces to
		\[\tau_\mu (\chi_\rho)= \dfrac{1}{2} \frac{\partial F}{\partial z}.\]	
		It is interesting that the adjoint torsion with respect to a meridian is related to a derivative of the defining equation of the character variety. A similar observation in terms of the A-polynomial was pointed out in \cite[Remark 4.5]{dimofte2013quantum}.
	\end{remark}
	
	\subsection{Global residue}
	We remark some relation of Theorem \ref{thm:main} with the global residue theorem. Recall that the global residue theorem says that any top-dimensional meromorphic form defined on a compact complex manifold has global residue zero.  Here global residue means the total sum of local residues. We refer to \cite{griffiths1978principles, tsikh1992multidimensional} for general references on residue theory.
	
	Theorem \ref{thm:main2} gives us that for generic $c \in \Cbb^\ast$  the level set $\mathrm{tr}_\gamma^{-1}(c+\frac{1}{c})$ is the common zero set $Z_{F,G}$ of $F(m,z)$ and $G(m,z):=E_\gamma(m,z)-c$ in $\Cbb^\ast \times \Cbb$ and 
	\[\sum_{\chi_\rho \in \mathrm{tr}_\gamma^{-1}(c+\frac{1}{c})} \frac{1} {\tau_\gamma(\chi_\rho)} =  \sum_{(m,z) \in Z_{F,G}} \,\dfrac{-2c}{m\, \det \left( \dfrac{\partial(F,G)}{\partial(m,z)}\right)}.
	\]  
	The above equation implies that Theorem \ref{thm:main} is equivalent to that the global residue of the meromorphic 2-form
	\[\omega = \dfrac{1}{F\cdot G} \frac{dm}{m} \wedge dz\]  defined on $\Cbb^\ast \times \Cbb$ is zero for generic $c\in\Cbb^\ast$.
	
	If $F$ and $G$ are generic in some sense, then there exists a toric compactification of $\Cbb^\ast \times \Cbb$ so that one can deduce that the global residue of $\omega$ is zero, as an application of the global residue theorem. We refer to \cite{khovanskii1978newton, vidras2001some} for precise conditions of genericness; the condition in \cite{vidras2001some} relaxes that in \cite{khovanskii1978newton}.
	Unfortunately, we however do not know a direct way to obtain Theorem \ref{thm:main} from \cite{vidras2001some}. Checking the generic condition in \cite{vidras2001some} seems to require heavy complex analysis  (see Remark \ref{rmk:question}).
	We thus take a practical detour to prove Theorem \ref{thm:main} by reducing the problem to a one-variable problem. We refer to Section \ref{sec:three} for details.

	\subsection*{Acknowledgment}
	The author is grateful to Jinsung Park and Anh Tran for their valuable discussions and comments. He also thanks the anonymous referee for careful reading.
	The author was supported by a KIAS Individual Grant (MG073801) at Korea
	 Institute for Advanced Study and Basic Science Research Program through the NRF of Korea funded by the Ministry of Education (2020R1A6A3A03037901).

	\section{Adjoint Reidemeister torsion of twist knots} \label{sec:two}
	
	\subsection{A brief review on definitions} \label{sec:rev}

	We briefly recall a definition of the adjoint torsion for a knot exterior. We refer to \cite{porti1997torsion,turaev2002torsions,dubois2003torsion} for more details.

	Let $C_\ast=(0 \rightarrow C_n \rightarrow \cdots \rightarrow C_0\rightarrow 0)$ be a chain complex of vector spaces over a field $\Fbb$ with a boundary map $\partial$.
	For a basis $c_\ast$ of $C_\ast$ and a basis  $h_\ast$  of the homology $H_\ast(C_\ast)$, the \emph{sign-refined Reidemeister torsion} $\mathrm{Tor}(C_\ast,c_\ast, h_\ast)$ is defined as follows.
	For each $0 \leq i \leq n$ we choose a tuple $b_i$ of vectors in $C_i$ such that $\partial b_i$ is a basis of $\partial C_i$, and a representative $\widetilde{h}_i$ of $h_i$ in $C_i$.
	Then one can check that a tuple $c'_i=(\partial b_{i+1}, \widetilde{h_i}, b_i)$ is a basis of $C_i$.
	Letting $A_i$ be the transition matrix taking the basis $c_i$ to the other basis $c'_i$, we have
	\begin{equation}\label{eqn:torsion}
	\mathrm{Tor}(C_\ast,c_\ast, h_\ast) :=  (-1)^{\sum_{j=0}^n \alpha_j \beta_j}\prod_{i=0}^n  \det A_i^{(-1)^{i+1}	} \in \Fbb^\ast
	\end{equation}
	where $\alpha_j = \sum_{i=0}^j \dim C_i$ and $\beta_j = \sum_{i=0}^j \dim H_i(C_\ast)$. Note that the equation \eqref{eqn:torsion} does not depend on the choices of $b_\ast$ and $\widetilde{h}_\ast$.
	
	Let $M$ be the knot exterior of a knot $K \subset S^3$. We fix any triangulation of $M$ and an orientation of each cell so that the cells, say $c_1,\cdots,c_m$, form a basis of $C_\ast(M;\Rbb)$. 
	It is well-known that $\dim H_i(M;\Rbb) =1$ for $i=0,1$ and $\dim H_i(M;\Rbb) =0$ otherwise. We choose a basis $h_\ast$ of $H_\ast(M;\Rbb)$ as  $h_\ast=\{[\mathrm{pt}],[\mu]\}$ where $\mathrm{pt}$ is a point in $M$ and $\mu$ is a meridian of $K$. Let
	\begin{equation}\label{eqn:tau}
	\epsilon= \mathrm{sgn} \left( \mathrm{Tor}(C_\ast(M;\Rbb),c_\ast,h_\ast) \right) \in  \{ \pm1\}
	\end{equation}
	where $\mathrm{sgn}(x)$ denotes the sign of $x \in \Rbb^\ast$.
	
	Let $\Mt$ be the universal cover of $M$ with the induced triangulation and $\rho :\pi_1(M)\rightarrow \G{2}$ be an irreducible representation.
	Viewing the Lie algebra $\mathfrak{g}$ of $\G{2}$ as a $\Zbb[\pi_1(M)]$-module through the adjoint representation $\mathrm{Ad}_\rho : \pi_1(M)\rightarrow \mathrm{Aut}(\mathfrak{g})$ associated to $\rho$, we consider a chain complex
	\[ C_*(M;\mathfrak{g}) = \mathfrak{g} \otimes_{\Zbb[\pi_1(M)]} C_*(\Mt;\Zbb)\]
	with a basis $\mathcal{C} = \{h,e,f\} \otimes \{\widetilde{c}_1 ,\cdots, \widetilde{c}_m \}$.
	Here  $\{h,e,f\}$ is a basis of $\mathfrak{g}$ and $\widetilde{c}_i$ is any lift of $c_i$ to $\Mt$.
	The chain complex $C_\ast(M;\mathfrak{g})$ usually has non-trivial homology $H_\ast(M;\mathfrak{g})$.
	We choose a boundary curve $\gamma \subset \partial M$ and fix a basis $\mathcal{H}$ of $H_\ast(M;\mathfrak{g})$ as follows under the assumption that $\rho$ is \emph{$\gamma$-regular} \cite{porti1997torsion}, i.e.,
	\begin{itemize}
		\item $\dim  H_1(M;\mathfrak{g})=1$; 
		\item the inclusion $\gamma \hookrightarrow M$  induces an epimorphism $H_1(\gamma;\mathfrak{g}) \rightarrow H_1(M;\mathfrak{g})$;
		\item if $\mathrm{tr} (\rho(\pi_1(\partial M)))  \subset \{ \pm 2\}$, then $\rho(\gamma) \neq \pm\mathrm{Id}$.
	\end{itemize} 
	\begin{remark} \label{rmk:regular} The $\gamma$-regularity is invariant under conjugating $\rho$ and thus the notion of $\gamma$-regular character is well-defined. Most of irreducible characters are $\gamma$-regular. Precisely, non-$\gamma$-regular irreducible characters are contained in the
		zero set of the differential of the trace function $\mathrm{tr}_\gamma : \mathcal{X}^\mathrm{irr}(M)\rightarrow \Cbb$. See \cite[Proposition 3.26]{porti1997torsion} and  \cite[Remark 9]{dubois2009non}. 
	\end{remark}
	
	From the Poincare duality, we have $\dim H_i(M;\mathfrak{g})=1$ for $i=1,2$ and $\dim H_i(M;\mathfrak{g})=0$ otherwise.
	We choose any non-zero element $v \in \mathfrak{g}$ invariant under  $\mathrm{Ad}_{\rho}(g)$ for all $g \in \pi_1(\partial M)$ and let
	$\mathcal{H}$ consist of the images of $v \otimes \gamma$ and $v \otimes \partial M$ under the canonical  maps $H_1(\gamma;\mathfrak{g}) \rightarrow H_1(M;\mathfrak{g})$ and $H_2(\partial M;\mathfrak{g}) \rightarrow H_2(M;\mathfrak{g})$, respectively.
	We finally define
	the \emph{adjoint torsion} $\tau_\gamma(\chi_\rho)$ as
	\begin{equation*} 
	\tau_\gamma(\chi_\rho) =\epsilon \cdot \mathrm{Tor}(C_\ast(M;\mathfrak{g}), \mathcal{C}, \mathcal{H}) \in \Cbb^\ast.
	\end{equation*}
	\begin{remark}
		The definition of $\tau_\gamma(\chi_\rho)$ involves several choices: a triangulation of $M$, an order/orientations of the cells of $M$, lifts of the cells of $M$ to $\Mt$, a basis of $\mathfrak{g}$, and the vector $v$. However, it turns out that the adjoint torsion does not depend on these choices. Also, the definition of $\tau_\gamma(\chi_\rho)$ involves the choices of orientations of $K$, $\mu$, $\gamma$, and $\partial M$. If we reverse one of these, the sign of $\tau_\gamma(\chi_\rho)$ changes.
		We therefore fix these orientations once and for all.
		Note that choices of these orientations do not matter when we consider Conjecture \ref{conj:main}.	
	\end{remark}
	\subsection{Adjoint torsions of twist knots} \label{sec:comp}
	
	We fix $n \neq 0 \in \Zbb$ and let $M$ be the knot exterior of the twist knot $K_{2n}$. Let $\mathcal{X}^\mathrm{irr}(M)$ be the character variety of irreducible representations $\pi_1(M)\rightarrow \G{2}$. Recall that $\mathcal{X}^\mathrm{irr}(M)$, as a set, is the set of conjugacy classes of irreducible representations \cite{culler1983varieties}: 
	\[\mathcal{X}^\mathrm{irr}(M) = \left\{ \rho : \pi_1(M)\rightarrow \mathrm{SL}_2(\Cbb) : \mathrm{irreducible} \right\}/_\mathrm{Conjugation}.\]
	 It is well-known that the fundamental group of $M$ has a presentation
	\[\pi_1(M)= \langle a,b \, | \, w^n a = b w^n \rangle \quad \textrm{where } w=ba^{-1}b^{-1}a\]
	and that up to conjugation an irreducible representation $\rho:\pi_1(M)\rightarrow \mathrm{SL}_2(\Cbb)$ is given by
	\begin{equation}\label{eqn:conj}
	\rho(a) = \begin{pmatrix}
	m & 1 \\ 0 &m^{-1}
	\end{pmatrix} \quad  \textrm{and} \quad
	\rho(b) = \begin{pmatrix}
	m & 0 \\ -u &m^{-1}
	\end{pmatrix}
	\end{equation} where $(m,u) \in (\Cbb^\ast)^2$ is a zero of the \emph{Riley polynomial} $R(m,u)$ \cite{riley1984nonabelian}.
	It is also proved in \cite{riley1984nonabelian} that two zeros $(m_1,u_1)$ and $(m_2,u_2)$ of the Riley polynomial represent the same character if and only if $m_1= m_2^{\pm1}$ and $u_1=u_2$. It follows that 
	\[\mathcal{X}^\mathrm{irr}(M) = \{ (m,u) \in (\Cbb^\ast)^2 : R(m,u)=0\}/_{(m,u) \sim (m^{-1},u)}.\]
	
	Let $S_k(z) \in \Cbb[z]$ be the Chebyshev polynomials defined by $S_{k+1}(z) = z S_k(z) - S_{k-1}(z)$ for all $k \in \Zbb$ with the initial condition $S_0(z)=0$, $S_1(z)=1$. 
	Then the Riley polynomial is explicitly given as 
		\begin{equation} \label{eqn:e1}
		R(m,u)=S_{n+1}(z)-(u^2-(u+1)(m^2+m^{-2}-3)) S_{n}(z) \in \Cbb[m^{\pm1},u]
		\end{equation}
		where $z$ is the trace of $\rho(w)$
		\begin{equation} \label{eqn:e2}
		z = \mathrm{tr} \rho(w) = u^2+2u +2 -(m^2+m^{-2})u.
		\end{equation}
		We refer to  \cite[\S 2.2]{tran2018twisted} for details.
		Note that the index of the Chebyshev polynomials in this paper is slightly different  from  \cite{tran2018twisted}; $S_k(z)$ in this paper agrees with $S_{k+1}(z)$ defined in \cite{tran2018twisted}.
	In what follows, we write $S_k(z)$ simply as $S_k$ for notational convenience.
			Note that  the Chebyshev polynomials satisfy	$S_k^2 -  z S_k S_{k-1} + S_{k-1}^2=1$ for all $k \in \Zbb$.
	
	\begin{lemma} \label{lem:u} The variable $z$ determines the variable $u$  uniquely as
		\begin{equation} \label{eqn:u}
		u = \dfrac{(z-2)S_{n}}{S_{n}-S_{n-1}}.
		\end{equation}
	\end{lemma}
	\begin{proof} From the equation \eqref{eqn:e2} we have $m^2+m^{-2} =u+2+(2-z)u^{-1}$. Plugging it into the equation \eqref{eqn:e1}, we obtain
		\begingroup
		\allowdisplaybreaks
		\begin{align*}
		R(m,u)&=S_{n+1}-(u^2-(u+1)(m^2+m^{-2}-3)) S_{n}\\
		&=S_{n+1} - (u^2-(u+1)(u-1+(2-z)u^{-1} ) S_n\\
		&=S_{n+1} -(z-1 + (z-2)u^{-1} )S_n\\
		&=S_n-S_{n-1} -(z-2)S_nu^{-1}.
		\end{align*}
		\endgroup
		On the other hand, on the zero set of $R(m,u)$,
		both $S_{n}-S_{n-1}$ and $(z-2)S_{n}$ can not be zero, as they have no common zero. Therefore, $u$ is determined as in the equation \eqref{eqn:u}.
	\end{proof}

	Using Lemma \ref{lem:u}, one can use the variable $z$ instead of the variable $u$. In particular, we may describe $\mathcal{X}^\mathrm{irr}(M)$ in terms of variables $m$ and $z$ as follow.
	\begin{lemma} One has
	\[\mathcal{X}^\mathrm{irr}(M) = \left\{ (m,z) \in \Cbb^\ast \times \Cbb : F(m,z)=0 \right\}/_{(m,z) \sim (m^{-1},z)}\]
where
\begin{equation}\label{eqn:F}
F(m,z)=S_{n} (S_{n}-S_{n-1})(m^2+m^{-2}) - zS_n(S_n-S_{n-1})+1 \in \Cbb[m^{\pm 1},z].
\end{equation}
	\end{lemma}
\begin{proof} From the equation \eqref{eqn:u} we have
	\begin{equation}\label{eqn:us}
		u+1 = \frac{(z-1)S_n -S_{n-1}}{S_n-S_{n-1}}=\frac{S_{n+1}-S_n}{S_n-S_{n-1}} \quad \textrm{and} \quad u+2 =\frac{z S_{n}-2S_{n-1}}{S_n-S_{n-1}}=\frac{S_{n+1}-S_{n-1}}{S_n-S_{n-1}}.
	\end{equation}
	We then have
	\begingroup
	\allowdisplaybreaks
	\begin{align*}
		R(m,u)&=S_{n+1}-(u^2-(u+1)(m^2+m^{-2}-3)) S_{n}\\
		&=S_n(u+1)(m^2+m^{-2}) + S_{n+1}-S_n -(u+1)(u+2)S_n\\
		&=\frac{S_{n+1}-S_n}{S_n-S_{n-1}}\left( S_n(m^2+m^{-2}) + S_{n}-S_{n-1} - \frac{(zS_n-2S_{n-1})S_n}{S_n-S_{n-1}}\right)\\
		&=\frac{S_{n+1}-S_n}{S_n-S_{n-1}} \left(S_n(m^2+m^{-2}) - \frac{(z-1)S_n^2-S_{n-1}^2}{S_n-S_{n-1}}\right)\\
		&=\frac{S_{n+1}-S_n}{S_n-S_{n-1}} \left(S_n(m^2+m^{-2}) - \frac{z S_n(S_n-S_{n-1})-1}{S_n-S_{n-1}}\right)\\
				&=\frac{S_{n+1}-S_n }{(S_n-S_{n-1})^2} F(m,z).
	\end{align*}
	\endgroup
	Note that we used the fact $S_n^2 - z S_n S_{n-1} +S_{n-1}^2=1$  in the fifth equation.
	 
	On the other hand, if $S_{n+1}-S_{n}=0$, then we have $u=-1$ from the equation \eqref{eqn:us} and  $z=m^2+m^{-2}+1$ from the equation \eqref{eqn:e2}. It then follows that
	\begingroup
	\allowdisplaybreaks
			\begin{align*}
		F(m,z)&=S_n(S_n-S_{n-1})(m^2+m^{-2}) - z S_n(S_n-S_{n-1})+1\\
		&=S_n(S_n-S_{n-1})(z-1) - z S_n(S_n-S_{n-1})+1	\\
		&=-S_n(S_n-S_{n-1})+1\\
		&=-(z-1) S_nS_{n-1}+S_{n-1}^2\\
		&=-S_{n-1}(S_{n+1}-S_{n}) =0.
		\end{align*}
	\endgroup
	Note that we used the fact $S_n^2 - zS_n S_{n-1} +S_{n-1}^2=1$  in the  forth equation. This completes the proof of $R(m,u)=0 \Leftrightarrow F(m,z)=0$.
	\end{proof}
%
	\begin{remark}\label{rmk:simp} The equation $F(m,z)=0$ implies that $m^2+m^{-2} + f(z)=0$ for a certain rational function $f(z)$. It allows us to simplify
 any element $g\in\Cbb[m^{\pm2},z]$ into of the form $A(z) m^2+ B(z)$.
  For instance, we have
		\begin{align*}
			A m^4 + B m^2+ C &= -A  m^2(f + m^{-2}) +B m^2+C \\
			&=(B-Af)m^2+C-A.
		\end{align*}
		This would be useful when we try to verify $g_1(m,z) = g_2(m,z) \in \Cbb[m^{\pm2},z]$; we first simplify them into $g_1(m,z) = A_1(z)m^2+B_1(z)$ and $g_2 = A_2(z)m^2+B_2(z)$ by using $F(m,z)=0$ and then check one-variable problems $A_1(z) = A_2(z)$ and $B_1(z) = B_2(z)$.
	\end{remark}
	We fix a point $(m,z) \in \Cbb^\ast \times \Cbb$ satisfying $F(m,z)=0$ and let $\rho : \pi_1(M)\rightarrow \G{2}$ be the irreducible representation given as in the equation \eqref{eqn:conj}. We denote  by $\chi_\rho$ the character of $\rho$.
	We choose a meridian $\mu = a$ and let $\lambda$ be the canonical longitude with respect to $\mu$.
	A simple formula for computing the adjoint torsion $\tau_\lambda(\chi_\rho)$ with respect to $\lambda$ is given in \cite{tran2014twisted}. We express the formula in terms of the variable $z$ as follows.
	
	\begin{theorem}\label{thm:tran} The adjoint torsion $\tau_\lambda(\chi_\rho)$ is given by
		\[\tau_\lambda(\chi_\rho) = -(2n+1)S_{n}^2+2n S_{n}S_{n-1} -   \frac{2((n-1)S_n+n S_{n-1})}{(z+2)(S_{n}-S_{n-1})}\]
		if $\chi_\rho$ is $\lambda$-regular.
	\end{theorem}
	\begin{proof} 
		It is computed in \cite[Corollary 2.6]{tran2014twisted} that 		
		\begin{equation} \label{eqn:tran}\tau_\lambda(\chi_\rho)=\dfrac{-1}{(y+2-x^2)(y^2-yx^2+x^2)}  \left(\dfrac{(2n-1)y^2+yx^2-2nx^2(x^2-2)}{y^2-yx^2+2x^2}+2n\right)
		\end{equation}
		where $x=m+m^{-1}$ and $y=u+2$.
		A straightforward computation gives that (recall the equation \eqref{eqn:us})
		\begingroup
		\allowdisplaybreaks
		\begin{align*} 
			y+2-x^2 &=u+2-(m^2+m^{-2}) \\
			&=\frac{z S_{n}-2S_{n-1}}{S_n-S_{n-1}} - \frac{ zS_n(S_n-S_{n-1})-1}{S_n(S_n-S_{n-1})}\\
			&=1-\frac{S_{n-1}}{S_n} =: 1- \alpha
		\end{align*}
		\endgroup
		Note that we used $F(m,z)=0$ and $S_n^2 - z S_n S_{n-1}+S_{n-1}^2 =1$  in the second and third equations, respectively.
		It follows from $x^2 =y+1 -\alpha$ that
		\begin{align*}
			&(2n-1)y^2+yx^2 -2n x^2(x^2-2) \\
			&=(2n-1)y^2+y(y+1 + \alpha) - 2n(y+1+\alpha)(y-1+\alpha) \\
			&=((1-4n)\alpha+1)y + 2n(1-\alpha^2) \\
			&=((1-4n)\alpha+1) \frac{z S_{n}-2S_{n-1}}{S_n-S_{n-1}} + 2n(1-\alpha^2)\\
			&=((1-4n)\alpha+1) \frac{z -2 \alpha}{1-\alpha} + 2n(1-\alpha^2).
		\end{align*}		
		 Similar computations gives  
		\begin{align*}
		 y^2-yx^2+x^2 = \frac{1}{S_n^2(1-\alpha)} \quad \textrm{and} \quad y^2-yx^2+2x^2 =z+2.
		\end{align*}
		Plugging above equations into the equation \eqref{eqn:tran}, we have
		\begin{align*}
			\tau_\lambda(\chi_\rho)  &= -S_n^2 \left( \frac{((1-4n)\alpha+1)(z -2 \alpha)}{(1-\alpha)(z+2)} + \frac{2n(1-\alpha^2)}{z+2}+2n\right)\\
			&=-(2n+1)S_n^2 + 2n S_n S_{n-1} -  \frac{2((n-1)S_n+n S_{n-1})}{(z+2)(S_{n}-S_{n-1})}.
		\end{align*}
		Note that the last equation is obtained from the fact $S_n^2- z S_nS_{n-1}+S_{n-1}^2=1$.
	\end{proof}
		\begin{remark} One can easily check that $S_k$ has value $(-1)^{k+1} k$ at $z=-2$ for all $k \in \Zbb$. It follows that $(n-1)S_n + nS_{n-1}$ has a zero at $z=-2$ and thus $(n-1)S_n+n S_{n-1}$ is divided by $z+2$.	Precisely, it is known that $(k-1)S_k + k S_{k-1} = (z+2)(S_k'-S'_{k-1})$ for all $k \in \Zbb$ where $S_k'$ is the derivative of $S_k$.
	\end{remark}

	Let $\gamma \subset \partial M$ be a boundary curve of slope $p/q \in \mathbb{Q} \cup \{\infty\}$. Namely, $\gamma = \mu^p\lambda^q$ for coprime integers $p$ and $q$. 
	We need eigenvalues of $\rho(\gamma)$ in order to  compute the adjoint torsion $\tau_\gamma(\chi_\rho)$.

	\begin{lemma} \label{lem:gamma} 
		The matrix $\rho(\gamma)$ is of the form 
		\[\rho(\gamma)= \rho(\mu^p\lambda^q)=\begin{pmatrix}
		m^pl^q & \ast \\ 0 & m^{-p} l^{-q}
		\end{pmatrix}\]
		where 
		\begin{equation} \label{eqn:longitude}
		l= -\dfrac{(z-2)(S_{n+1}-S_{n-1})S_{n}^2}{S_{n}-S_{n-1}}  m^2+ (z-2)(S_{n}+S_{n-1})S_{n}+1.
		\end{equation}
	\end{lemma}
	\begin{proof} It is enough to show that the $(1,1)$-entry of $\rho(\lambda)$ coincides with the given $l$. It is known that $\lambda=w_\ast^n w^n$ where $w_\ast$ is the word obtained by writing $w$ in the reversed order (see e.g. \cite{riley1984nonabelian}).
		From \cite[\S 3.2]{anh2016reidemeister} we have
		\begin{align*}
		\rho(w^n)&=\begin{pmatrix}
		S_{n+1}-(1+(2-m^{-2})u+u^2)  S_{n} & (m^{-1}-m-mu)  S_{n}\\
		((m-m^{-1})u+mu^2)  S_{n} & S_{n+1}-(1-m^2u) S_{n}
		\end{pmatrix},\\
		\rho(w_\ast^n)&= \begin{pmatrix}
		S_{n+1}-(1-m^{-2}u)  S_{n} & (m-m^{-1}-m^{-1}u)  S_{n}\\
		((m^{-1}-m)u+m^{-1}u^2)  S_{n} & S_{n+1} -(1+(2-m^2)u+u^2) S_{n}
		\end{pmatrix}.
		\end{align*}
		It follows that the $(1,1)$-entry of $\rho(\lambda)=\rho(w^n_\ast)\rho(w^n)$ is 
		\begin{align*}
		&(S_{n+1}-(1-m^{-2}u) S_{n}) \cdot (S_{n+1}-(1+(2-m^{-2})u+u^2) S_{n}) \\
		& \hspace{3em} +  (m-m^{-1}-m^{-1}u) S_{n} \cdot  ((m-m^{-1})u+mu^2)  S_{n}\\
		&=S_{n+1}^2-(z+(m^2-m^{-2})u)S_{n}S_{n+1} \\
		& \hspace{3em} + (u(u+1)m^2-(u^3+u^2-1)-u(u^2+u+1)m^{-2}+u^2m^{-4})S_{n}^2\\
		&=1 - (m^2-m^{-2}) u S_{n}S_{n+1} \\
		& \hspace{3em} + (u(u+1)m^2-(u^3+u^2-1)-u(u^2+u+1)m^{-2}+u^2m^{-4}-1)S_{n}^2 \\
		&= 1 - \left(m+m^{-1}\right)\left((u+1)m-m^{-1}\right)(z-2)S_{n}^2.
		\end{align*}
		Note that we used the fact $S_{n+1}^2-zS_{n+1} S_{n} +S_{n}^2=1$ and the equation \eqref{eqn:u} ($\Leftrightarrow u S_n= ((z-1)u +z-2) S_{n-1}$) for the third and fourth equations, respectively. 
		Then the desired expression \eqref{eqn:longitude} is obtained by plugging $u+1=\frac{S_{n+1}-S_{n-1}}{S_n-S_{n-1}}$ (recall the equation \eqref{eqn:us}) and  then simplifying it as in Remark \ref{rmk:simp}.
	\end{proof}
	We define a function $E_\gamma$ on the zero set of $F$ in $\Cbb^\ast \times \Cbb$ as
	\begin{equation} \label{eqn:eigen2}
	E_\gamma(m,z) =  m^p \left(-\dfrac{(z-2) (S_{n+1}-S_{n-1})S_{n}^2}{S_{n}-S_{n-1}}  m^2+ (z-2)(S_{n}+S_{n-1})S_{n}+1\right)^q.
	\end{equation}
	It is clear from Lemma \ref{lem:gamma} that $E_\gamma(m,z)$ is an eigenvalue of $\rho(\gamma)$.
	\begin{remark}
		A similar computation given as in the proof of Lemma \ref{lem:gamma} shows that
		\begin{equation*} 
		l^{-1}= -\dfrac{(z-2) (S_{n+1}-S_{n-1})S_{n}^2}{S_{n}-S_{n-1}} m^{-2}+ (z-2)(S_{n}+S_{n-1})S_{n}+1
		\end{equation*}
		and thus $E_\gamma (m^{-1},z) = E_\gamma(m,z)^{-1}$.
	\end{remark}  
	%
	\begin{remark} \label{rmk:laruent}
		Using the equation $F(m,z)=0$, one can re-express
		\begin{align*}
		l&= -(z-2)S_{n}^2m^{4}-(z-2)(S_{n}-S_{n-1})m^{2}+(S_{n}-S_{n-1})^2\\
		l^{-1}&= -(z-2)S_{n}^2m^{-4}-(z-2)(S_{n}-S_{n-1})m^{-2}+(S_{n}-S_{n-1})^2
		\end{align*}
		so that $l^{\pm1}$ become Laurent polynomials. In particular, using these expressions, $E_\gamma$ becomes a Laurent polynomial.
	\end{remark}
	\begin{theorem} \label{prop:main}  The adjoint torsion $\tau_\gamma(\chi_\rho)$ is given by 
		\begin{equation}\label{eqn:thm}
			 \tau_\gamma(\chi_\rho) = -\dfrac{m}{2 E_\gamma}  \, \mathrm{det} \left( \dfrac{\partial(F,E_\gamma)}{\partial(m,z)}\right)
		\end{equation}
		if $\chi_\rho$ is $\gamma$-regular.
	\end{theorem}
	Note that as a consequence, the right-hand side of the equation \eqref{eqn:thm} has the same value at $(m,z)$ and $(m^{-1},z)$.

	\begin{proof}
		For simplicity let $F(m,z)=f_1(z)(m^2+m^{-2})+f_2(z)$ and  $E_\gamma(m,z)=m^p(g_1(z)m^2+g_2(z))^q$. See the equations \eqref{eqn:F} and \eqref{eqn:eigen2}. Also, we let $f_3=f_2/f_1$ so that $m^2+m^{-2}+f_3(z)=0$ and $dz/dm = -2(m-m^{-3})/f_3'$. 
		Straightforward computations give that
		\begin{align*}
		\dfrac{d l}{dm} &= \dfrac{d}{d m}(g_1 m^2+g_2) \\
		&= 2 g_1m +(g_1'm^2+g_2')\dfrac{d z}{d m}  \\
		&=2 g_1m - \dfrac{2(m-m^{-3})(g_1'm^2+g_2')}{f'_3}
		\end{align*}
		and that
		\begingroup
		\allowdisplaybreaks
		\begin{align}
		& \mathrm{det} \left( \dfrac{\partial(F/f_1,E_\gamma)}{\partial(m,z)}\right)  \nonumber \\
		&=  \mathrm{det} \left( \dfrac{\partial(m^2+m^{-2}+f_3,\ m^p(g_1 m^2+g_2)^q)}{\partial(m,z)}\right)  \nonumber \\
		&= \det \begin{pmatrix} 2(m-m^{-3}) & pm^{p-1}(g_1m^2+g_2)^q+qm^p(g_1m^2+g_2)^{q-1} \cdot 2g_1m \\[3pt] f_3'& q m^p(g_1m^2+g_2)^{q-1} (g_1'm^2+g_2')\end{pmatrix} \nonumber\\
		&= \det \begin{pmatrix} 2(m-m^{-3}) &  \frac{p E_\gamma}{m}+\frac{q E_\gamma}{l} \cdot 2g_1m \\[3pt] f_3'& \frac{q E_\gamma}{l} (g_1'm^2+g_2')\end{pmatrix} \nonumber\\
		&=-\dfrac{p E_\gamma}{m} f'_3 + \dfrac{q E_\gamma}{l} \left(2(m-m^{-3})(g'_1m^2+g_2')-2g_1f_3'm\right)\nonumber\\ 
		&=\dfrac{2E_\gamma}{l} \left((m-m^{-3})(g'_1m^2+g_2')-g_1f_3'm\right) \left( p \dfrac{l}{m} \dfrac{d m}{d l}+q\right) \nonumber \\
		&=\dfrac{2E_\gamma}{l}((-g'_1f_3+2g'_2-g_1f_3')m+(-2g_1'+g_2'f_3)m^{-1})\left( p \dfrac{l}{m} \dfrac{d m}{d l}+q\right). \label{eqn:det}
		\end{align}	
		\endgroup
		Note that the last equation is obtained from the equation $m^2+m^{-2}+f_3=0$ as in Remark \ref{rmk:simp}.
		
%
		
		We now claim that 
		\begin{equation}\label{eqn:zeqn}
			-g_1'f_3+2g_2'-g_1f_3' + g_1 \tau_\lambda(\chi_\rho)/f_1 = 0  \textrm{ and }  -2g_1'+g_2'f_3 + g_2 \tau_\lambda(\chi_\rho)/f_1 =0.
		\end{equation}
		Note these equations are only in the variable $z$ due to Theorem \ref{thm:tran}.
		From the identity $S_n^2 -z S_n S_{n-1}+S_{n-1}^2=1$, we have 
		\[2S_n S_n' - S_n S_{n-1} - z S_n'S_{n-1} - z S_n S_{n-1}'+ 2 S_{n-1} S_{n-1}'=0.\]
		Together with the identity $n S_{n-1} + (n-1)S_{n}= (z+2)(S'_{n}-S'_{n-1})$, we obtain
		\begin{equation}\label{eqn:zprime}
			S'_n=\dfrac{(n-1)z S_n-2nS_{n-1}}{z^2-4} \quad \textrm{and} \quad S'_{n-1}=\dfrac{2 (n-1) S_n-z n S_{n-1}}{z^2-4}.
		\end{equation}
		Recall that $g_1,g_2,f_1,f_2,$ and $f_3$ are given in terms of $S_{n-1},S_{n}$, and $S_{n+1}$.
		Plugging $S_{n+1}=zS_n -S_{n-1}$, they are given in terms of $S_{n-1}$ and $S_{n}$ and thus the equation \eqref{eqn:zeqn} is given in terms of $S_{n-1},S_{n},S_{n-1}',$ and $S_n'$.
		Then plugging the equation \eqref{eqn:zprime} into the equation \eqref{eqn:zeqn}, we obtain two equations, each of which consists of terms in $S_n$ and $S_{n-1}$. With the aid of Mathematica, one can check that both equations have a factor $S_n^2- z S_{n}S_{n-1}+S_{n-1}^2-1$ which is identically zero. This proves the equation \eqref{eqn:zeqn}.
		
		Combining the equations \eqref{eqn:det} and \eqref{eqn:zeqn}, we have
		\begin{align*}
		\mathrm{det} \left( \dfrac{\partial(F/f_1,E_\gamma)}{\partial(m,z)}\right) &= -\dfrac{2 E_\gamma}{l} \left( \dfrac{g_1 \tau_\lambda(\chi_\rho)}{f_1} m +  \dfrac{g_2 \tau_\lambda(\chi_\rho)}{f_1} m^{-1} \right)  \left( p \dfrac{l}{m} \dfrac{d m}{d l}+q\right)\\
		&=-\dfrac{2 E_\gamma}{l} \left( \dfrac{l}{mf_1}  \tau_\lambda(\chi_\rho) \right)  \left( p \dfrac{l}{m} \dfrac{d m}{d l}+q\right) \\
		&=-\dfrac{2E_\gamma}{mf_1} \, \tau_\gamma(\chi_\rho).
		\end{align*}
		Recall that the last equation follows from \cite[Theorem 4.1]{porti1997torsion}: 
		\[\tau_\gamma(\chi_\rho) = \tau_\lambda(\chi_\rho) \, \dfrac{d \log (m^pl^q)}{d \log l}= \tau_\lambda(\chi_\rho)  \left( p \dfrac{l}{m} \dfrac{d m}{d l}+q\right).\]
		This completes the proof, since we have
		\[
		\mathrm{det} \left( \dfrac{\partial(F/f_1,E_\gamma)}{\partial(m,z)}\right)=\dfrac{1}{f_1}\mathrm{det} \left( \dfrac{\partial(F,E_\gamma)}{\partial(m,z)}\right)\]
		on the zero set of $F(m,z)$.
	\end{proof}
	\begin{remark} In this paper, we use two variables for computational simplicity. However, using three variables with the variable $l$ seems also natural. Precisely, if we let $E_\gamma=m^pl^q$ and
		\[H(m,z,l)=l - \left(- \dfrac{(z-2) (S_{n+1}-S_{n-1})S_{n}^2}{S_{n}-S_{n-1}} m^2 + (z-2)(S_{n}+S_{n-1})S_{n}+1\right),\]
		then we have		\[ \tau_\gamma(\chi_\rho)= -\dfrac{m}{2E_\gamma}   \mathrm{det} \left( \dfrac{\partial(F,E_\gamma,H)}{\partial(m,z,l)}\right)\]
		for $(m,z,l) \in \Cbb^\ast \times \Cbb \times \Cbb^\ast$ satisfying $F=H=0$.			
	\end{remark}
	\section{Proof of Theorem \ref{thm:main}} \label{sec:three}
	We fix $n \in \Zbb \setminus \{0,1\}$ and let $M$ be the knot exterior of the twist knot $K_{2n}$. We also fix a boundary curve  $\gamma \subset \partial M$ of slope $p/q \in \mathbb{Q} \cup \{\infty\}$. We may assume that $q \geq 0$.
	In Section \ref{sec:two}, we computed that for generic $c \in \Cbb^\ast$
	\begin{equation} \label{eqn:target} 
	\sum_{\chi_\rho \in \mathrm{tr}_\gamma^{-1}(c+\frac{1}{c})} \frac{1} {\tau_\gamma(\chi_\rho) } =\sum_{(m,z)\in Z_{F,G}} \dfrac{-2c}{m \det \left( \dfrac{\partial(F,G)}{\partial(m,z)}\right)}
	\end{equation}
	where 
	\begin{align*}
	F(m,z)&=S_{n}(S_{n}-S_{n-1})(m^2+m^{-2})- zS_n(S_n-S_{n-1})+1 \\
	G(m,z)&=m^p \left(-\dfrac{(z-2) (S_{n+1}-S_{n-1})S_{n}^2}{S_{n}-S_{n-1}}  m^2+ (z-2)(S_{n}+S_{n-1})S_{n}+1\right)^q-c
	\end{align*}	
	and $Z_{F,G}$ is the common zero set of $F$ and $G$ in $\Cbb^\ast \times \Cbb$. Note that genericity of $c$ guarantees that $\mathrm{tr}_\gamma^{-1}(c+\frac{1}{c})$ consists of $\gamma$-regular characters (see Remark \ref{rmk:regular}).

	\begin{remark} \label{rmk:question} Recall Remark \ref{rmk:laruent} that we can take $G$ as a Laurent polynomial. 
		If the system $(F,G)$ of Laurent polynomials is $(\Delta_F,0)$-\emph{proper} (see \cite{vidras2001some} for the definition), where $\Delta_F$ is the Newton polygon of $F$, then Theorem \ref{thm:main} is directly obtained from \cite[Theorem 1.2]{vidras2001some}. 
	\end{remark}
	\subsection{For even $p$} \label{sec:even}
	For simplicity let $F(m,z)=f_1(z)(m^2+m^{-2})+f_2(z)$, $G(m,z)=m^p(g_1(z)m^2+g_2(z))^q-c$, and $f_3(z)=f_2(z)/f_1(z)$.
	Note that by definition if $f_1(z_0)=0$ for some $z_0 \in \Cbb$, then $f_2(z_0)=1$. In particular, $f_1$ and $f_2$ has no common zero.

	From the equation $F/f_1=m^2+m^{-2}+f_3=0$, we recursively obtain (see  Remark \ref{rmk:simp})
	\[m^{2k} = h_k(z) m^2 - h_{k-1}(z)\]
	for all $k \in \Zbb$ where $h_0(z)=0$, $h_1(z)=1$, and $h_{k+1}(z)=-f_3(z) h_k(z) - h_{k-1}(z)$. 
	It follows that
	\begin{align*}
		G(m,z) &= m^p (g_1 m^2+g_2)^q-c \\
		&=\sum_{k=0}^q \binom{q}{k} g_1^k  g_2^{q-k}m^{2k+p} -c \\
		&=\left(\sum_{k=0}^q \binom{q}{k} g_1^k  g_2^{q-k} h_{k+\frac{p}{2}}\right) m^2 - \left(\sum_{k=0}^q \binom{q}{k} g_1^k  g_2^{q-k} h_{k+\frac{p}{2}-1}+c\right) \\
		&=: \alpha(z)  m^2  - \beta(z)
	\end{align*} on the zero set of $F(m,z)$.

	\begin{lemma} \label{lem:lem1} For generic $c \in \Cbb^\ast$ one has
			\[ Z_{F,G} = \left\{(m,z) \in \Cbb^\ast \times \Cbb : H(z) =0,\  f_1(z) \neq 0,\ m = \pm \sqrt{\dfrac{\beta(z)}{\alpha(z)}} \right\}\]
		where $H=f_1(\frac{\alpha}{\beta}+\frac{\beta}{\alpha})+f_2$. 
	\end{lemma}
	\begin{proof} 
		
		For $(m_0,z_0)\in Z_{F,G}$ we have $f_1(z_0) \neq 0$; otherwise, we have $f_1(z_0)=f_2(z_0)=0$ which contradicts the fact that $f_1$ and $f_2$ have no common zero. Since denominators of $f_3, g_1,$ and $g_2$ are $f_1=S_n(S_n-S_{n-1})$, $S_n-S_{n-1}$, and $1$, respectively,  denominators of $\alpha$ and $\beta$ divide some power of $f_1$. Thus the fact $f_1(z_0) \neq 0$ implies that both $\alpha$ and $\beta$ have no pole at $z_0$. If either $\alpha(z_0)=0$ or $\beta(z_0)=0$, then from $\alpha(z_0) m_0^2 - \beta(z_0)=0$ we have $\alpha(z_0)=\beta(z_0)=0$  which fails for generic $c$. Therefore, both $\alpha(z_0)$ and $\beta(z_0)$ are non-zero complex numbers. It follows that $m_0 = \pm \sqrt{\beta(z_0)/\alpha(z_0)}$ and $H(z_0)=0$.
		
		Conversely, let $z_0$ be a zero of $H$ with $f_1(z_0)\neq0$. Recall that $f_1(z_0) \neq 0$ implies that both $\alpha$ and $\beta$ have no pole at $z_0$. For generic $c$, we may assume that $\alpha$ and $\beta$ have no common zero. Then from $H(z_0)=0$ we have $\alpha(z_0) \neq 0$ and $\beta(z_0) \neq 0$; otherwise, $H$ has a pole at $z_0$.  Letting $m_0=\pm \sqrt{\beta(z_0)/\alpha(z_0)}$, it is clear that $m_0$ is a non-zero complex number satisfying $F(m_0,z_0)=0$ and $G(m_0,z_0)=0$.
	\end{proof}

	\begin{lemma} \label{lem:det}
	For each $(m,z) \in Z_{F,G}$ one has
	\[ \mathrm{det} \left( \dfrac{\partial(F,G)}{\partial(m,z)}\right) = -2m \alpha H'.\]
\end{lemma}
\begin{proof} Recall that we have $m^{2k}=h_km^2-h_{k-1}$ from the equation $F(m,z)=0$. It follows that 
	\begin{equation}\label{eqn:detzero}
	\mathrm{det} \left( \dfrac{\partial(F,m^{2k}-h_{k}m^2+h_{k-1})}{\partial(m,z)}\right) =0
	\end{equation} for all $k\in\Zbb$.
	We then have
	\begin{align*}
		\mathrm{det} \left( \dfrac{\partial(F,G)}{\partial(m,z)}\right) &=  \mathrm{det} \left( \dfrac{\partial(F,\alpha m^2 -\beta)}{\partial(m,z)}\right) =  \mathrm{det} \left( \dfrac{\partial(H,\alpha m^2 -\beta)}{\partial(m,z)}\right)=-2m\alpha H'.
	\end{align*}
Here the first equation follows from the equation \eqref{eqn:detzero} together with the fact that $\alpha m^2 -\beta$ differs from $G$ by some linear combination of $m^{2k}-h_km^2 + h_{k-1}$. Similary, the second equation follows from the fact that $H$ differs from $F$ by some multiple of $\alpha m^2 -\beta$.
\end{proof}

Rewriting the equation \eqref{eqn:target} by using Lemmas \ref{lem:lem1} and \ref{lem:det}, we have
\begingroup
\allowdisplaybreaks
\begin{align}
	\sum_{\chi_\rho \in \mathrm{tr}_\gamma^{-1}(c+\frac{1}{c})} \frac{1} {\tau_\gamma(\chi_\rho) } &=\sum_{(m,z)\in Z_{F,G}} \, \dfrac{-2c}{m\, \mathrm{det} \left( \dfrac{\partial(F,G)}{\partial(m,z)}\right)} \nonumber \\
	&= \sum_{ \substack {H(z)=0 \\ f_1(z) \neq 0}}\dfrac{c}{m^2 \alpha H'} \quad \mathrm{where}\ m =\pm\sqrt{\frac{\beta(z)}{\alpha(z)}}\nonumber  \\
	&= \sum_{ \substack { H(z)=0 \\ f_1(z) \neq 0}} \dfrac{2c}{ \beta H'} \nonumber  \\
	&=\sum_{H(z)=0} \dfrac{2c}{ \beta H'} \label{eqn:a}
\end{align}
\endgroup
where the last equation follows from Lemma \ref{lem:lem2} below.
	\begin{lemma} \label{lem:lem2} $\beta H'$ has a pole at each common zero of $H$ and $f_1$ for generic $c \in \Cbb^\ast$. 
	\end{lemma}
	\begin{proof}
	
		Let $z_0$ be a common zero of $H$ and $f_1$ and let 	$\nu$ be the discrete valuation counting the order of zero/pole at $z_0$. 	Clearly, we have $\nu(H)>0$ and thus $\nu(H') <\nu(H)$. 
		If $\nu(\beta H)\leq 0$, then we have $\nu(\beta H')<\nu(\beta H) \leq 0$ and we are done. In what follows, we prove that the other case $\nu(\beta H)>0$ gives us contradiction.
			
		Let $\delta = \beta-c$ so that $\delta$ does not depend on $c$. From the fact that $E_\gamma(m,z)= \alpha m^2 -\delta$ on the zero set of $F$ and  $E_\gamma(m,z)E_\gamma(m^{-1},z)=1$, we have 	\begin{align}
			1 &= (\alpha m^2 - \delta)(\alpha m^{-2} - \delta) \nonumber\\
			&=\alpha^2 + \delta^2 - \alpha \delta(m^2+m^{-2}) \nonumber \\
			&=\alpha^2 +\delta^2 +f_3 \alpha \delta \label{eqn:ad}
		\end{align}
		on the zero set of $F$. 
		Since $\alpha^2+\delta^2+f_3 \alpha \delta$ is a one-variable rational function, it should be identically 1.

		Recall that by definition we have $\alpha \beta H = f_1(\alpha^2+\beta^2)+f_2 \alpha \beta$.
		From the equation \eqref{eqn:ad} $\Leftrightarrow f_1=f_1(\alpha^2+\delta^2)+f_2 \alpha \delta$, we have
		\begin{align}
			\alpha \beta H - f_1 &=f_1(\beta^2-\delta^2)+f_2 \alpha(\beta-\delta) \nonumber\\
		\Leftrightarrow	\alpha \beta H - f_1 &=f_1( 2 \delta c +c^2)+f_2 \alpha c 	\nonumber\\
		\Leftrightarrow c^{-1}\beta H  & =2  \alpha^{-1} f_1 \delta +\alpha^{-1}f_1(c+c^{-1})+f_2 . \label{eqn:nu} 
		\end{align}
		Recall that $f_1(z_0)=0$ implies $f_2(z_0)=1$. From the equation \eqref{eqn:nu} we conclude that for generic $c$ the only possible case of $\nu(\beta H)>0$ is that $\nu(\delta)<0$ and $\nu(f_1)+\nu(\delta) = \nu(\alpha)$ with $(\alpha^{-1}f_1 \delta)(z_0)=-\frac{1}{2}$. On the other hand, from the equation \eqref{eqn:ad} we have
		\begin{align*}
			f_1 &= f_1(\alpha^2+\delta^2) +f_2 \alpha \delta\\
			\Leftrightarrow \alpha^{-1}f_1 &= f_1 \alpha +\delta (\alpha^{-1}f_1  \delta +f_2). 
			\end{align*}
	Since $(\alpha^{-1}f_1 \delta +f_2)(z_0)=-\frac{1}{2}+1 =\frac{1}{2}$
	 and $\nu(f_1 \alpha ) = 2\nu(f_1)+\nu(\delta) > \nu(\delta)$, we have $\nu(f_1 \alpha + \delta (\alpha^{-1}f_1 \delta +f_2)) = \nu(\delta)$. It follows that $\nu(\alpha^{-1}f_1) = \nu(\delta)$ and $-\nu(\delta) =\nu(\delta)$ ($\because \nu(f_1)+\nu(\delta)=\nu(\alpha)$). This contradicts to $\nu(\delta)<0$.
	\end{proof}

	We finally claim that the equation \eqref{eqn:a} is zero due to Jacobi's residue theorem. Recall that Jacobi's residue theorem says that any non-constant polynomial $f$ with $f(0)\neq 0$ and no multiple zero satisfies
	\[\sum_{f(z)=0} \dfrac{g(z)}{f'(z)} =0\]
	for any polynomial $g$ with $\deg g \leq \deg f- 2$. 	
	\begin{lemma} \label{lem:jac} For generic $c\in\Cbb^\ast$ one has
		\[	\sum_{ H(z)=0} \dfrac{1}{ \beta H'}=0.\] 
	\end{lemma}
	\begin{proof} Recall that $\alpha$ and $\beta$ are rational functions in the variable $z$. We let $\alpha =\alpha_1/\alpha_2$ and $\beta=\beta_1/\beta_2$ for some polynomials $\alpha_1,\alpha_2,\beta_1$, and $\beta_2$ with $\mathrm{gcd}(\alpha_1,\alpha_2)=\mathrm{gcd}(\beta_1,\beta_2)=1$.	 
	Here we take the greatest common divisor in $\Cbb[z]$.
	Plugging these into $H$, we have
		\[H=\frac{f_1(\alpha_1^2 \beta_2^2 + \alpha_2^2 \beta_1^2)+f_2 \alpha_1\alpha_2 \beta_1 \beta_2}{\alpha_1 \alpha_2 \beta_1\beta_2}=:\frac{H_1 d}{H_2 d}\]
	where $\mathrm{gcd}(H_1,H_2)=1$ and $d$ is the greatest common divisor of the numerator and denominator.
	Clearly, we have
	\[\sum_{ H(z)=0} \dfrac{1}{ \beta H'}=\sum_{ H_1(z)=0} \dfrac{1}{ \frac{\beta_1}{\beta_2}  (\frac{H_1}{H_2})'}=\sum_{H_1(z)=0} \dfrac{\beta_2 H_2}{ \beta_1 H_1'}.\]
	Note that $\mathrm{gcd}(\beta_1, H_1d)=\mathrm{gcd}(\beta_1,f_1 \alpha_1^2\beta_2^2)=1$ for generic $c$ ($\because f_1 \alpha_1^2 \beta_2$ does not depend on $c$). In particular, $\beta_1$ and $d$ have no common zero. It follows from $H_2 d = \alpha_1\alpha_2 \beta_1 \beta_2$ that $\beta_1$ divides $H_2$ and thus  $H_2/\beta_1$ is a polynomial.
	
	For generic $c$ we may assume that $H_1$ has no multiple zero and is non-zero at $z=0$. 
	Lemma \ref{lem:deg} below shows that $\deg \beta_2 H_2 - \deg \beta_1 \leq  \deg  H_1 - 2$. Then we have
	\[	\sum_{H_1(z)=0} \dfrac{\beta_2 H_2/\beta_1}{H_1'} =0\]
	from Jacobi's residue theorem.
	\end{proof}

	\begin{lemma} \label{lem:deg} For generic $c \in \Cbb^\ast$ one has $\deg H + \deg \beta \geq 2$.
	\end{lemma}
	Here the degree of a rational function means the degree difference of the numerator and denominator: for instance, $\deg H= \deg H_1 - \deg H_2$ and $\deg \beta = \deg \beta_1 - \deg \beta_2$.
	\begin{proof} Recall the proof of Lemma \ref{lem:lem2} that  $\alpha^2+\delta^2 + f_3 \alpha \delta=1$ where $\delta = \beta-c$. Since $\deg f_3= \deg f_2 - \deg f_1=1$, we have $\deg \alpha \neq \deg \delta $. On the other hand, one can simplify $H$ as 
		\begin{align*}
			H &= \frac{f_1(\alpha^2+\beta^2)+f_2 \alpha \beta}{\alpha \beta} \\
			&=\frac{f_1(\alpha^2 + \delta^2+2c\delta +c^2) + f_2 \alpha (\delta+c)}{\alpha \beta} \\
			&=\frac{f_1(1+2c\delta +c^2) + f_2 \alpha c}{\alpha \beta} \\
			&=\frac{cf_1 (c^{-1}+c  + 2\delta +f_3 \alpha)}{\alpha \beta} \\			
			&=\frac{cf_1 ( (c^{-1}+c) \delta +1-\alpha^2+\delta^2)}{\alpha \beta \delta}.
		\end{align*}
	It follows that for generic $c$ 
	\[		\deg H +\deg \beta  = \deg f_1 + \mathrm{Max}(\deg \delta,\, \deg( 1-\alpha^2+\delta^2)) -\deg \alpha -\deg \delta . \]
	If $\deg \alpha <0$, then \[\deg H+\deg \beta \geq \deg f_1 + \deg \delta - \deg \alpha -\deg \delta = \deg f_1 - \deg \alpha \geq \deg f_1 +1.\] If $\deg \alpha \geq 0$, then 	
	\begin{align*}
		\deg H +\deg \beta  
		 \geq \deg f_1 + \deg(1-\alpha^2+\delta^2)-\deg \alpha -\deg \delta  \geq \deg f_1+1.
	\end{align*}
	Note that the last inequality follows from the fact $\deg \alpha \neq \deg \delta$. This completes the proof, since $\deg f_1 \geq 1$ for $n \neq 0,1$.
	
	\end{proof}
	
	\subsection{For odd $p$} As in Section \ref{sec:even}, we have
	\begin{align*}
	G(m,z) &= m^{-1}m^{p+1} (g_1 m^2+g_2)^q-c\\
	&=\left(\sum_{k=0}^q \binom{q}{k} g_1^k  g_2^{q-k} h_{k+\frac{p+1}{2}}\right) m - \left(\sum_{k=0}^q \binom{q}{k} g_1^k  g_2^{q-k} h_{k+\frac{p-1}{2}}\right)m^{-1}-c \\
	&=: \alpha(z)  m  - \beta(z) m^{-1}-c.
	\end{align*}
	on the zero set of $F$. From the fact that $E_\gamma(m,z) E_\gamma(m^{-1},z)=1$, we have
	\begin{align*}
		1 &= (\alpha m - \beta m^{-1})(\alpha m^{-1} - \beta m) \\
		&=\alpha^2 + \beta^2 - \alpha \beta (m^2+m^{-2}) \\
		&=\alpha^2 +\beta^2 +f_3 \alpha \beta
	\end{align*}
	on the zero set of $F$ and therefore $\alpha^2+\beta^2+f_3 \alpha \beta$ should be identically 1.

	It is not hard to solve the equations $F(m,z)=0$ and $\alpha m -\beta m^{-1} -c=0$ for generic $c \in \Cbb^\ast$ (see the proof of Lemma \ref{lem:det2} below for details):
	\[ Z_{F,G} = \left\{(m,z) \in \Cbb^\ast \times \Cbb : H(z) =0,\ m = \dfrac{c\alpha + \beta/c}{\alpha^2-\beta^2} \right\}\]
	where
	$H= (f_3+2)(\alpha+\beta)^2+(c-1/c)^2$.
	\begin{remark}  $H$ and $f_1$ have no common zero for generic $c$. 
	It follows that for every zero $z_0$ of $H$, $f_3(z_0)+2$ is a certain complex number and thus $\alpha(z_0)+\beta(z_0) \neq 0$. On the other hand, using the equation $\alpha^2+\beta^2+f_3\alpha \beta=1$, we can rewrite $H$ as
	\begin{equation}
	H=(f_3-2)(\alpha-\beta)^2+(c+1/c)^2.
	\end{equation} We then have $\alpha(z_0)-\beta(z_0)  \neq 0$ similarly. It follows that $m=\frac{c \alpha(z_0)+\beta(z_0)/c}{\alpha^2(z_0)-\beta^2(z_0)}$ is a non-zero complex number.
	\end{remark}
	\begin{lemma} \label{lem:det2}
		For each $(m,z) \in Z_{F,G}$ one has
		\[ \mathrm{det} \left( \dfrac{\partial(F,G)}{\partial(m,z)}\right) = \frac{cf_1}{(\beta^2-\alpha^2)m} H'.\]
	\end{lemma}
	\begin{proof} 
		Using the Euclidean algorithm, one can find an equation, which is linear in the variable $m$, among linear combinations of $F$ and $B:=\alpha m -\beta m^{-1}-c$. Precisely, we have
		\begin{align*}
		D &:= (\alpha \beta^2f_1^{-1})  F + (- \beta^2 m +\alpha \beta m^{-1}-c \alpha	) B \\
		&= c(\beta^2-\alpha^2) m + c^2 \alpha + \beta(\alpha^2 + \beta^2+ f_3 \alpha \beta )\\
		&= c(\beta^2-\alpha^2) m + c^2 \alpha + \beta.
		\end{align*}
		In particular, we have
		\begin{align*}
		\mathrm{det} \left( \dfrac{\partial(F,G)}{\partial(m,z)}\right)& =
		\mathrm{det} \left( \dfrac{\partial(F,B)}{\partial(m,z)}\right) \\ &=\det \begin{pmatrix} \alpha \beta ^2f_1^{-1} & 0 \\ \ast & 1 \end{pmatrix}^{-1} \mathrm{det} \left( \dfrac{\partial(D,B)}{\partial(m,z)}\right)
		\end{align*}
		where the first equation follows from the equation \eqref{eqn:detzero} (as in the proof of Lemma \ref{lem:det}).
		Plugging the equation $D$ (which is linear in $m$) into $B$ to eliminate the variable $m$, we obtain
		\begin{align*}
		E&:=\dfrac{\alpha  \beta ^2 ((f_3+2)(\alpha+\beta)^2+(c-1/c)^2)}{\left(\alpha ^2-\beta ^2\right) \left(c\alpha +\beta/c  \right)}.
		\end{align*}
		Note that we used the equation $\alpha^2+\beta^2 + f_3\alpha \beta =1$ in the simplification.
	 	It follows that 
		\[
		\mathrm{det} \left( \dfrac{\partial(D,B)}{\partial(m,z)}\right)= \mathrm{det} \left( \dfrac{\partial(D,E)}{\partial(m,z)} \right) 
		= c(\beta^2-\alpha^2) E'.
		\]
		Combining the above computations, we have
		\begin{align*} \mathrm{det} \left( \dfrac{\partial(F,G)}{\partial(m,z)}\right)&=\dfrac{f_1}{\alpha \beta^2}\mathrm{det} \left( \dfrac{\partial(D,B)}{\partial(m,z)}\right)\\
		&=\dfrac{cf_1 (\beta^2-\alpha^2)}{\alpha \beta^2} E'\\
		&=-\dfrac{cf_1}{c\alpha  +\beta/c } H'\\
		&=\dfrac{cf_1}{(\beta^2-\alpha^2)m}H'.
		\end{align*}
		Note that the third and fourth equations follow from the equations $H=0$ and $D=0$, respectively.
	\end{proof}
	Rewriting the equation \eqref{eqn:target} by using Lemma \ref{lem:det2}, we have
	\[	\sum_{\chi_\rho \in \mathrm{tr}_\gamma^{-1}(c+\frac{1}{c})} \frac{1} {\tau_\gamma(\chi_\rho)} = \sum_{(m,z)\in Z_{F,G}} \dfrac{-2c}{m\ \mathrm{det} \left( \dfrac{\partial(F,G)}{\partial(m,z)}\right)} =\sum_{ H(z)=0} \dfrac{2(\alpha^2-\beta^2)}{ f_1 H'}.\]
	As in Section \ref{sec:even}, we claim that the above equation is zero due to Jacobi's residue theorem.
	\begin{lemma} For generic $c \in \Cbb^\ast$ one has
				\[	\sum_{H(z)=0} \dfrac{\alpha^2-\beta^2}{ f_1 H'}=0.\] 
	\end{lemma}
		\begin{proof} Let $\alpha = \alpha_1/\alpha_2$ and $\beta= \beta_1/\beta_2$ for some polynomials $\alpha_1,\alpha_2,\beta_1$, and $\beta_2$ with $\mathrm{gcd}(\alpha_1,\alpha_2)=\mathrm{gcd}(\beta_1,\beta_2)=1$.	 
		Plugging these into $H$, we have
		\[H=\frac{(f_2+2f_1)(\alpha_1 \beta_2 + \alpha_2 \beta_1)^2+(c-1/c)^2 f_1 \alpha_2^2 \beta_2^2}{f_1\alpha_2^2 \beta_2^2}=:\frac{H_1d}{H_2d}\]
		with $\mathrm{gcd}(H_1,H_2)=1$. Then we have
		\begin{equation}\label{eqn:last}	\sum_{  H(z)=0} \dfrac{\alpha^2-\beta^2}{ f_1 H'}=\sum_{H_1(z)=0} \dfrac{(\alpha_1^2 \beta_2^2 - \alpha_2^2 \beta_1^2)H_2}{f_1 \alpha_2^2 \beta_2^2 H_1'} = \sum_{H_1(z)=0} \frac{(\alpha_1^2 \beta_2^2 - \alpha_2^2 \beta_1^2)/d}{ H_1'}.
		\end{equation}

		We claim that $d$ divides $\alpha_1^2 \beta_2 ^2 - \alpha_2^2 \beta_1^2$ and thus $(\alpha_1^2 \beta_2^2 -\alpha_2^2 \beta_1^2)/d$ is a polynomial. It follows from the definitions of $\alpha$ and $\beta$ that the denominators $\alpha_2$ and $\beta_2 $ divide some power of $f_1$. In particular, every zero $z_0$ of $d$ is also a zero of $f_1$.
		Let $\nu$ be the discrete valuation counting the order of zero/pole at $z_0$. 
		From the fact that $d$ divides $(f_2+2f_1)(\alpha_1 \beta_2 + \alpha_2 \beta_1)^2=(H_1-(c-1/c)^2H_2)d$ and that $\mathrm{gcd}(f_1,f_2)=1$, we have
		\[\nu(d) \leq 2\nu(\alpha_1 \beta_2+ \alpha_2 \beta_1).\]
		On the other hand, using the equality $\alpha^2+\beta^2+f_3\alpha \beta=1$, one can rewrite $H$ as
		\[H=\frac{(f_2-2f_1)(\alpha_1 \beta_2 - \alpha_2 \beta_1)^2+(c+1/c)^2 f_1 \alpha_2^2 \beta_2^2}{f_1\alpha_2^2 \beta_2^2}=\frac{H_1d}{H_2d}\]
		and then we have
		$\nu(d) \leq 2\nu(\alpha_1 \beta_2 - \alpha_2 \beta_1)$ similarly.
	Therefore,
		\[\nu(d) \leq \nu(\alpha_1 \beta_2 + \alpha_2 \beta_1)+\nu(\alpha_1 \beta_2 - \alpha_2 \beta_1) = \nu(\alpha_1^2 \beta_2^2 -\alpha_2^2 \beta_1^2).\] 
		This proves that $d$ divides $\alpha_1^2 \beta_2^2 -\alpha_2^2 \beta_1^2$.
		
		Lemma \ref{lem:deg2} below shows that $\deg(\alpha_1^2\beta_2^2 -\alpha_2^2 \beta_1^2) - \deg d \leq \deg H_1-2$, which is equivalent to $\deg (\alpha^2-\beta^2) \leq \deg H+\deg f_1 -2$ (see the equation \eqref{eqn:last}). Then the lemma follows from Jacobi's residue theorem.
		\end{proof}
	\begin{lemma} \label{lem:deg2}For generic $c \in \Cbb^\ast$ one has $\deg H + \deg f_1 -2 \geq \deg (\alpha^2 -\beta^2)$.
	\end{lemma}
	\begin{proof} 
		Recall that $f_1H=(f_2+2f_1)(\alpha+\beta)^2+(c-1/c)^2 f_1$. It follows that for generic $c \in \Cbb^\ast $ \[\deg f_1 + \deg H \geq  \mathrm{Max}(\deg (f_2+2f_1)(\alpha+\beta)^2, \deg f_1).\] In particular, $\deg f_1+ \deg H \geq \deg (f_2+2f_1) + 2 \deg (\alpha+\beta)$. 
		One can easily check that $\deg (f_2 +2f_1) \geq 2$ for $n\neq 0,1$. Also, we have  $\deg \alpha \neq \deg \beta$ from the equation $\alpha^2+\beta^2 + f_3 \alpha \beta=1$ with $\deg f_3=1$. It follows that   
	 $\deg (\alpha + \beta) = \deg (\alpha - \beta)$. This completes the proof, as we have
	 \begin{align*}
	 		\deg f_1 + \deg H &\geq \deg (f_2+2f_1) + 2 \deg (\alpha+\beta) \\
	 						& \geq 2 + \deg(\alpha + \beta) + \deg (\alpha
	 						 -\beta)\\
	 						&= 2 +\deg(\alpha^2-\beta^2).
	 \end{align*}
	\end{proof}

%
	
	\begin{remark} Lemmas \ref{lem:deg} and \ref{lem:deg2} do not hold for $n=1$ (the trefoil knot). Furthermore, the equation \eqref{eqn:target} is non-zero for the trefoil knot.	
		This shows that the hyperbolicity assumption in Conjecture \ref{conj:main} is essential.
	\end{remark}

	\bibliographystyle{gtart}
	\bibliography{biblog}

\begin{thebibliography}{}
\providecommand\bibmarginpar{\leavevmode\marginpar}
\def\urlstyle#1{{\tt #1}}

\bibitem{benini2019rotating}
\textbf{F Benini}, \textbf{D Gang}, \textbf{L\,A\,P Zayas}, \emph{{Rotating
  Black Hole Entropy from M5 Branes}}, preprint arXiv:1909.11612

\bibitem{cooper1994plane}
\textbf{D Cooper}, \textbf{M Culler}, \textbf{H Gillet}, \textbf{D Long},
  \textbf{P\,B Shalen}, \emph{Plane curves associated to character varieties of
  3-manifolds}, Inventiones mathematicae 118 (1994) 47--84

\bibitem{culler1983varieties}
\textbf{M Culler}, \textbf{P\,B Shalen}, \emph{Varieties of group
  representations and splittings of 3-manifolds}, Annals of Mathematics  (1983)
  109--146

\bibitem{dimofte2013quantum}
\textbf{T Dimofte}, \textbf{S Garoufalidis}, \emph{The quantum content of the
  gluing equations}, Geometry \& Topology 17 (2013) 1253--1315

\bibitem{dubois2003torsion}
\textbf{J Dubois}, \emph{Torsion de Reidemeister non ab{\'e}lienne et forme
  volume sur l'espace des repr{\'e}sentations du groupe d'un noeud}, PhD
  thesis, Universit{\'e} Blaise Pascal-Clermont-Ferrand II (2003)

\bibitem{dubois2009non}
\textbf{J Dubois}, \textbf{V Huynh}, \textbf{Y Yamaguchi}, \emph{{Non-abelian
  Reidemeister torsion for twist knots}}, Journal of Knot Theory and Its
  Ramifications 18 (2009) 303--341

\bibitem{gang2019precision}
\textbf{D Gang}, \textbf{N Kim}, \textbf{L\,A\,P Zayas}, \emph{{Precision
  microstate counting for the entropy of wrapped M5-branes}}, preprint
  arXiv:1905.01559

\bibitem{gang2019adjoint}
\textbf{D Gang}, \textbf{S Kim}, \textbf{S Yoon}, \emph{{Adjoint Reidemeister
  torsions from wrapped M5-branes}}, preprint arXiv:1911.10718

\bibitem{griffiths1978principles}
\textbf{P Griffiths}, \textbf{J Harris}, \emph{Principles of algebraic
  geometry}, John Wiley \& Sons (1978)

\bibitem{hatcher1985incompressible}
\textbf{A Hatcher}, \textbf{W Thurston}, \emph{Incompressible surfaces in
  2-bridge knot complements}, Inventiones mathematicae 79 (1985) 225--246

\bibitem{khovanskii1978newton}
\textbf{A\,G Khovanskii}, \emph{{Newton polyhedra and the Euler-Jacobi
  formula}}, Russian Mathematical Surveys 33 (1978) 237--238

\bibitem{macasieb2011character}
\textbf{M Macasieb}, \textbf{K Petersen}, \textbf{R van Luijk}, \emph{On
  character varieties of two-bridge knot groups}, Proceedings of the London
  Mathematical Society 103 (2011) 473--507

\bibitem{menasco1984closed}
\textbf{W Menasco}, \emph{Closed incompressible surfaces in alternating knot
  and link complements}, Topology 23 (1984) 37--44

\bibitem{porti1997torsion}
\textbf{J Porti}, \emph{{Torsion de Reidemeister pour les vari{\'e}t{\'e}s
  hyperboliques}}, volume 612, American Mathematical Soc. (1997)

\bibitem{riley1984nonabelian}
\textbf{R Riley}, \emph{Nonabelian representations of 2-bridge knot groups},
  The Quarterly Journal of Mathematics 35 (1984) 191--208

\bibitem{tran2014twisted}
\textbf{A Tran}, \emph{{Twisted Alexander polynomials with the adjoint action
  for some classes of knots}}, Journal of Knot Theory and Its Ramifications 23
  (2014) 1450051

\bibitem{anh2016reidemeister}
\textbf{A Tran}, \emph{{Reidemeister torsion and Dehn surgery on twist knots}},
  Tokyo J. Math 39 (2016)

\bibitem{tran2018twisted}
\textbf{A Tran}, \emph{{Twisted Alexander polynomials of genus one two-bridge
  knots}}, Kodai Mathematical Journal 41 (2018) 86--97

\bibitem{tsikh1992multidimensional}
\textbf{A\,K Tsikh}, \emph{Multidimensional residues and their applications},
  volume 103, Providence, RI: American Mathematical Society (1992)

\bibitem{turaev2002torsions}
\textbf{V\,G Turaev}, \emph{Torsions of 3-manifolds}, Geometry \& Topology
  Monographs 4 (2002)

\bibitem{vidras2001some}
\textbf{A Vidras}, \textbf{A Yger}, \emph{{On some generalizations of Jacobi's
  residue formula}}, from ``Annales scientifiques de l'Ecole normale
  sup{\'e}rieure'', volume~34 (2001)  131--157

\bibitem{witten1989quantum}
\textbf{E Witten}, \emph{Quantum field theory and the Jones polynomial},
  Communications in Mathematical Physics 121 (1989) 351--399

\bibitem{yamaguchi2008relationship}
\textbf{Y Yamaguchi}, \emph{{A relationship between the non-acyclic
  Reidemeister torsion and a zero of the acyclic Reidemeister torsion}}, from
  ``Annales de l'institut Fourier'', volume~58 (2008)  337--362

\end{thebibliography}
\end{document}